\documentclass[11pt, reqno]{amsart}
\usepackage{amssymb, amsthm, amsmath, amsfonts}
\usepackage{graphics}
\usepackage{hyperref}
\usepackage{bbm}
\usepackage[all]{xy}
\usepackage{enumerate}
\usepackage[mathscr]{eucal}
\usepackage{enumerate}
\usepackage{thmtools}
\usepackage{xcolor}
\usepackage{ytableau}

\theoremstyle{plain}
\newtheorem{theorem}{Theorem}[section]
\newtheorem{lemma}[theorem]{Lemma}
\newtheorem{proposition}[theorem]{Proposition}

\newtheorem{corollary}[theorem]{Corollary}
\numberwithin{equation}{section}

\theoremstyle{definition}

\newtheorem{question}[theorem]{Question}
\newtheorem{definition}[theorem]{Definition}
\newtheorem{example}[theorem]{Example}
\newtheorem{remark}[theorem]{Remark}

\newcommand{\Mod}{{\textrm{-}\mathrm{Mod}}}
\newcommand{\Inj}{{\mathrm{Inj}}}

\DeclareMathOperator{\Aut}{Aut}
\DeclareMathOperator{\Sym}{Sym}
\DeclareMathOperator{\Homeo}{Homeo}
\DeclareMathOperator{\Hom}{Hom}

\DeclareMathOperator{\Stab}{Stab}

\newcommand{\OI}{{\mathrm{OI}}}
\newcommand{\FI}{{\mathrm{FI}}}

\newcommand{\op}{{\mathrm{op}}}

\newcommand{\dis}{{\mathrm{dis}}}
\newcommand{\Sh}{{\mathrm{Sh}}}

\title{Noetherianity of polynomial rings up to group actions}

\author{Liping Li}
\address{LCSM (Ministry of Education), School of Mathematics and Statistics, Hunan Normal University, Changsha 410081, China.}
\email{lipingli@hunnu.edu.cn}

\author{Yinhe Peng}
\address{Academy of Mathematics and Systems Science, Chinese Academy of Sciences\\ East Zhong Guan Cun Road No. 55\\Beijing 100190\\China}
\email{pengyinhe@amss.ac.cn}

\author{Zhengjun Yuan}
\address{School of Mathematics and Statistics, Hunan Normal University, Changsha 410081, China.}
\email{zhengjunyuan@hunnu.edu.cn}

\thanks{L. Li was partly supported by NSFC Grant No. 12171146. The authors wish to express their sincere gratitude to Andrew Snowden for providing detailed clarifications concerning some overlaps between this work and \cite{LS} as well as a simpler proof of Corollary \ref{highly homogenous result}, and to Shurui Yuan for suggestions that have contributed to the simplification of certain arguments. The authors are also indebted to the anonymous reviewer for a thorough examination of the manuscript and for numerous valuable comments, which have led to significant improvements in the present paper.}


\begin{document}

\begin{abstract}
Let $k$ be a commutative Noetherian ring, and $k[S]$ the polynomial ring whose indeterminates are parameterized by elements in a set $S$. We show that $k[S]$ is Noetherian up to highly homogenous actions of groups. In particular, there is a special linear order $\leqslant$ on infinite $S$ such that $k[S]$ is Noetherian up to actions of $\Aut(S, \leqslant)$, and the existence of such a linear order for every infinite set is equivalent to the axiom of choice. These Noetherian results are proved via a sheaf theoretic approach based on Artin's theorem, the work of Nagel-R\"{o}mer in \cite{NR1}, and a classification of highly homogenous groups by Cameron in \cite{Ca76} .
\end{abstract}

\maketitle

\section{Introduction}

\subsection{Motivation}

Let $k$ be a commutative Noetherian ring, $S$ a set, and let $k[S] = k[x_s \mid s \in S]$ be the polynomial ring whose indeterminates are parameterized by elements in $S$. It is well known that $k[S]$ is Noetherian if and only if $S$ is a finite set. However, sometimes for practical purpose people need to study subposets of ideals satisfying extra properties, and ask whether the ascending chain condition holds for these subposets. Specifically, given a group $G$ acting on $k[S]$, is the poset of all $G$-invariant ideals of $k[S]$ Noetherian? Of course, when $G$ is the trivial group, this question has been completely answered by Hilbert's basis theorem. But for general groups, only sporadic results exist in the literature, among which includes the following classical theorem established and generalized by a few authors including Cohen, Aschenbrenner-Hillar, and Hillar-Sullivant: the poset of $\Sym(S)$-invariant ideals of $k[S]$ is Noetherian, where $\Sym(S)$ is the full permutation group on $S$. For details, see \cite{AH, Co, Dr, HS}.

Motivated by Cohen's theorem, people began to consider structures and behaviors of ideals invariant under actions of groups, and a significant progress along this approach has been achieved. For instances, Nagpal-Snowden \cite{NS} obtained a complete classification of all $\Sym(S)$-invariant prime ideals of $k[S]$ as well as an explicit description of its equivariant spectrum; Laudone-Snowden \cite{LS} proved an analogue of Cohen's theorem for chains of ideals parameterized by various finite combinatorial structures; Yu \cite{Y} proved several finiteness results for $k[S]$-modules invariant under the action of the parabolic subgroup of the infinite general linear group $\mathrm{GL} (\mathbb{C}^{\infty})$.

The main goal of this paper is to generalize Cohen's theorem from the specific group $\Sym(S)$ to general groups (in particular, subgroups of $\Sym(S)$) for the following two reasons. Firstly, the full permutation group is rather big, and consequently, the subposet of $\Sym(S)$-invariant ideals of $k[S]$ is very small compared to the poset of all ideals. Thus we want to find subgroups $G \leqslant \Sym(S)$ such that a similar theorem is valid. Secondly and more importantly, it is often the case that $S$ is equipped with extra combinatorial, algebraic or topological structure, and in this case we are more interested in permutation groups preserving this structure. For example, let $S = \mathbb{R}$ be the set of real numbers, and let $G$ be the group of all permutations preserving the natural order on $\mathbb{R}$, or the group of all self-homeomorphisms on $\mathbb{R}$. It has been shown in \cite{LS} that the poset of $G$-invariant ideals of $k[\mathbb{R}]$ satisfies the ascending chain condition.

\subsection{Main results}

Before describing the main results of this paper, let us give necessary notations and definitions. Let $G$ be a group acting on $S$ via a group homomorphism $\rho: G \to \Sym(S)$. We equip $S$ with the discrete topology, and impose a topology on $G$ such that the family of pointwise stabilizers $\Stab_G(T)$, where $T$ is a finite subset of $S$, forms a cofinal system of open subgroups. With respect to this topology $G$ becomes a topological group, and the action of $G$ on $S$ is continuous.\footnote{In some literature $S$ is called a continuous $G$-set. However, since it is often the case that $S$ may have a natural topology (for instance, $S = \mathbb{R}$), to avoid possible confusion, in this paper we call $S$ a discrete $G$-set to emphasize that $S$ is equipped with the discrete topology.} In particular, if the action of $G$ on $S$ is faithful, or equivalently, $\rho$ is an injective homomorphism, then this topology on $G$ is precisely the \textit{compact-open topology} or pointwise convergence topology; see \cite{Are, Hu} for details.

Note that the action of $G$ on $S$ induces a natural action on the polynomial ring $k[S]$ equipped with the discrete topology, where elements of $G$ act on $k[S]$ as algebra automorphisms. Thus one can define the skew group ring $k[S]G$, whose multiplication is defined via the following formula
\[
(ag) \cdot (bh) = a (g \cdot b) gh
\]
for $a, b \in k[S]$ and $g, h \in G$, where $g \cdot b$ is the image of $b$ under the action of $g$. Given a $k[S]G$-module $V$, we say that $V$ is a \textit{discrete} $k[S]G$-module if the action of $G$ on $V$ is continuous with respect to the discrete topology on $V$. The category $k[S]G \Mod^{\dis}$ of discrete $k[S]G$-modules is abelian. As we will explain later, it is equivalent to a sheaf category over a ringed site, and hence is a Grothendieck category with enough injective objects.

For an integer $n \in \mathbb{N}$, let $S^n$ be the Cartesian product of $n$ copies of $S$, $S^{(n)}$ the set of termwise distinct $n$-tuples, and $S^{\{n\}}$ the set of $n$-element subsets of $S$. Clearly, the action of $G$ on $S$ induces natural actions of $G$ on these sets. Following \cite[Definition 3.11]{BMMN} or \cite[Section 2.1]{Ca}, we say that the action of $G$ on $S$ is \textit{highly transitive} if $S^{(n)}$ is a transitive $G$-set for every $n \in \mathbb{N}$, \textit{highly homogenous} if $S^{\{n\}}$ is a transitive $G$-set for every $n \in \mathbb{N}$, and \textit{oligomorphic} if $S^n$ is a union of finitely many transitive $G$-sets for every $n \in \mathbb{N}$.

Our first main result is:

\begin{theorem} \label{main result 1}
If the action of $G$ on $S$ via a group homomorphism $\rho: G \to \Sym(S)$ is highly transitive, then one has the following equivalence
\[
k[S]G \Mod^{\dis} \simeq k[S]\tilde{G} \Mod^{\dis}
\]
where $\tilde{G} = \Sym(S)$. Moreover, every finitely generated discrete $k[S]G$-module is Noetherian. In particular, $k[S]$ as a $k[S]G$-module is Noetherian.
\end{theorem}

\begin{remark}
Since $k[S]$ as a $k[S]G$-module is Noetherian if and only if the poset of $G$-invariant ideals is Noetherian, and since the action of $\Sym(S)$ on $S$ is obviously highly transitive, the last statement of the above theorem is equivalent to Cohen's theorem. We shall mention that this statement has already been established in \cite[Example 3.2(a)]{LS}.
\end{remark}

Now we impose some extra structure on $S$. As a first attempt, suppose that $S$ is equipped with an unbounded linear order $\leqslant$. We say that $\leqslant$ is a \textit{homogenous} linear order if for every two pairs $a < b$ and $c < d$ in $(S, \leqslant)$, one has $(a, b) \cong (c, d)$ as linearly ordered sets. It is \textit{globally homogeneous} if for every pair $a < b$ in $(S, \leqslant)$, one has $(a, b) \cong (S, \leqslant)$. Typical examples of globally homogeneous linearly ordered sets include $\mathbb{R}$ and $\mathbb{Q}$ with the natural order. For $\mathbb{R}$, this is clearly true; for $\mathbb{Q}$, this follows from the classical result due to Cantor: all countable dense linearly ordered sets without endpoints are isomorphic to $(\mathbb{Q}, \leqslant)$; see \cite[Theorem 9.3]{BMMN}. It is easy to see that $\leqslant$ is a homogenous linear order if and only if the action of $\Aut(S, \leqslant)$ on $(S, \leqslant)$ is highly homogenous.

\begin{theorem} \label{main result 2}
Let $\leqslant$ be a homogenous linear order on $S$, and suppose that the action of $G$ on $S$ via a group homomorphism $\rho: G \to \Aut(S, \leqslant)$ is highly homogenous. Then one has the following equivalence
\[
k[S]G \Mod^{\dis} \simeq k[S]\tilde{G} \Mod^{\dis}
\]
where $\tilde{G} = \Aut(S, \leqslant)$. Moreover, every finitely generated discrete $k[S]G$-module is Noetherian. In particular, $k[S]$ as a $k[S]G$-module is Noetherian.
\end{theorem}

\begin{remark}
The last statement of this theorem has been established in \cite[Example 3.2(c)]{LS}.
\end{remark}

Based on this result as well as a classification of highly homogenous groups obtained by Cameron in \cite{Ca76}, we can deduce the following result, extending Cohen's theorem from the highly transitive case to the highly homogenous case. This result was known to Snowden, who gave a simpler argument in a personal conversation with the authors.

\begin{corollary} \label{main corollary}
If the action of $G$ on $S$ via a group homomorphism $\rho: G \to \Sym(S)$ is highly homogenous, then $k[S]$ as a $k[S]G$-module is Noetherian.
\end{corollary}

Motivated by Theorem \ref{main result 2}, one may ask the following natural question: can we impose a homogeneous linear order or even a globally homogenous linear order on any nonempty set $S$? The answer is certainly no for finite sets. For infinite sets, we obtain the following surprising answer.

\begin{theorem} \label{main result 3}
The statement that every infinite set admits a globally homogeneous linear order is equivalent to the axiom of choice.
\end{theorem}

\subsection{Questions}

There are still quite a lot interesting questions unsolved in this paper.

For $n \in \mathbb{N}$, define function $f: \mathbb{N} \to \mathbb{N} \cup \{ \infty \}$ via setting $f(n)$ to be the number of $G$-orbits of $S^{(n)}$, called the \textit{growth function} of the pair $(G, S)$. It is easy to see that the action of $G$ on $S$ is highly transitive (resp., oligomorphic) if and only if $f(n) = 1$ (resp, $f(n) < \infty$) for $n \geqslant 1$. We can reformulate the first statement of Theorem \ref{main result 1} as follows: $k[S]G \Mod^{\dis}$ is equivalent to $k[S]\tilde{G} \Mod^{\dis}$ for $\tilde{G} = \Sym(S)$ if and only if the growth function for $(G, S)$ is the same as that for $(\tilde{G}, S)$. For the only if part, see Proposition \ref{characterization for FI}.

\begin{question}
Does this reformulation extend to the general case? Explicitly, given a group $G$ and a subgroup $H \leqslant G$ such that the action of $H$ on $S$ is oligomorphic, can we obtain a necessary and sufficient condition for the equivalence $k[S]G \Mod^{\dis} \simeq k[S]H \Mod^{\dis}$ in terms of a relation between the growth function for $(G, S)$ and that for $(H, S)$? As a specific example, the diagonal action of $\tilde{G} = \Sym(S)$ on $\tilde{S} = S \sqcup S$ is not highly transitive, but oligomorphic. Can we classify subgroups $G \leqslant \Sym(S)$ such that $k[\tilde{S}]G \Mod^{\dis} \simeq k[\tilde{S}] \tilde{G} \Mod^{\dis}$? In particular, does this equivalence hold if and only if the growth function for $(G, \tilde{S})$ asymptotically coincides with that for $(\tilde{G}, \tilde{S})$?

We shall mention that a similar question has been proposed in \cite[Remark 3.7]{LS}.
\end{question}

By Example \ref{example}, the conclusion of Corollary \ref{main corollary} fails for oligomorphic groups; on the other hand, the condition that $G$ is highly homogenous on $S$ is not necessary for the conclusion. Thus one may ask the following question:

\begin{question}
Find a combinatorial condition in terms of the action of $G$ on $S$, stronger than the oligomorphic condition but weaker than the highly homogeneous condition, which is equivalent to the Noetherianity of $k[S]$ up to the action of $G$.
\end{question}

The following question is related to Theorem \ref{main result 3}.

\begin{question}
Consider the following statements:
\begin{enumerate}
\item every infinite set admits a globally homogenous linear order;

\item every infinite set admits a homogenous linear order;

\item every infinite set admits a dense linear order.
\end{enumerate}
It is clear from the definitions that (1) implies (2). We will show that every homogeneous linear order is dense and unbounded in Lemma \ref{unbounded and dense}, so (2) implies (3). Theorem \ref{main result 3} asserts that (1) is equivalent to the axiom of choice. Furthermore, it is also known that (3) is not strong enough to imply the axiom of choice; see \cite[Theorem 1 and Corollary 13]{Gonzalez} or \cite{Pin}. One may wonder whether (2) implies the axiom of choice. At this moment it is still unknown to the authors.
\end{question}

\subsection{The strategy}

The technical strategy to establish the above mentioned Noetherian results is based on a few observations, a theorem of Artin connecting sheaves of modules over ringed atomic sites and discrete representations of topological groups, and the work of Nagel and R\"{o}mer \cite{NR1} on modules over the category $\FI$ and $\OI$ with varying coefficients.

Given a topological group $G$ with a cofinal system of open subgroups, one can construct an orbit category $\mathcal{O}_G$. For a commutative ring $A$ equipped with discrete topology such that elements of $G$ acts as algebra automorphisms and the action is continuous, one can define a sheaf $\mathcal{A}$ of commutative rings over the Grothendieck site $(\mathcal{O}_G, J_{at})$ where $J_{at}$ is the atomic Grothendieck topology. Similarly, given a discrete $AG$-module $V$, one can define a sheaf $\mathcal{V}$ of modules over the ringed site $(\mathcal{O}_G, J_{at}, \mathcal{A})$. Artin's theorem then asserts that the category of discrete $AG$-modules is equivalent to the category $\Sh(\mathcal{O}_G, J_{at}, \mathcal{A})$ of sheaves; see Theorem \ref{equivalence of categories}.

Now let $A = k[S]$ and $G$ a group acting on $S$. We show that the action of $G$ on $S$ is highly transitive if and only if under a mild assumption the orbit category $\mathcal{O}_G$ is isomorphic to the opposite category of the category $\FI_S$ whose objects are finite sets of $S$ and morphisms are injections; see Proposition \ref{characterization for FI}. In this case, an explicit computation shows that the structure sheaf $\mathcal{A}$ over $(\mathcal{O}_G, J_{at})$ coincides with the functor $\mathbf{X}^{\FI, 1}$ in \cite[Definition 2.17]{NR1}. Then Theorem \ref{main result 1} follows from \cite[Theorem 6.15]{NR1}.

Similarly, suppose that $\leqslant$ is a linear order on $S$. Then there exists a group whose action on $(S, \leqslant)$ is highly homogenous if and only if $\leqslant$ is homogenous, and if and only if under a mild assumption the orbit category $\mathcal{O}_G$ is isomorphic to the opposite category of the category $\OI_S$ whose objects are finite subsets of $S$ and morphisms are order-preserving injections; see Lemma \ref{elementary properties} and Proposition \ref{characterization of OI}. In this case, the structure sheaf $\mathcal{A}$ coincides with the functor $\mathbf{X}^{\OI, 1}$ in \cite[Definition 2.17]{NR1}. Then Theorem \ref{main result 2} follows from \cite[Theorem 6.15]{NR1} as well.

\subsection{Organization}

The paper is organized as follows. In Section 2 we describe some background knowledge on sheaf theory over ringed sites, and establish an equivalence between the category of discrete modules over the skew group ring $AG$ and the category of sheaves over the ringed site $(\mathcal{O}_G, J_{at}, \mathcal{A})$. Using this equivalence, we define notions such as finite generation and Noetherianity in these two categories, and describe some elementary properties. Actions of highly transitive groups on the polynomial ring is investigated in Section 4, where we prove Theorem \ref{main result 1}. In Section 5 we consider actions of order-preserving permutation groups on the polynomial ring, and prove Theorems \ref{main result 2} and \ref{main result 3} as well as some corollaries.

Throughout this paper all rings are unital rings with a multiplicative identity, modules are left modules, and functors are covariant functors unless otherwise specified. Composition of maps or morphisms is always from right to left.

\section{Preliminaries}

For the convenience of the reader, in this section we describe some background knowledge on orbit categories, Grothendieck topologies, and topos theory. The reader can refer to \cite{Ar, AGV, Car, Jo, KS, MM} for more details.

Let $G$ be a topological group, and let $\mathcal{H}$ be a \textit{cofinal system} of open subgroups of $G$; that is, for $H \in \mathcal{H}$ and $K \in \mathcal{H}$, one can find a member $L \in \mathcal{H}$ such that $L \leqslant H \cap K$. The \textit{orbit category} $\mathcal{O}_G$ of $G$ with respect to $\mathcal{H}$ is defined as follows: objects are left cosets $G/H$ with $H \in \mathcal{H}$, and morphisms from $G/H$ to $G/K$ are $G$-equivariant maps. It is well known that these maps are induced by elements $g \in G$. That is, every map $G/H \to G/K$ is of the form $\sigma_g: xH \mapsto xg^{-1}K$ where $g$ is an element in $G$ satisfying $gHg^{-1} \leqslant K$. The map $g \mapsto \sigma_g$ defines a surjective map
\[
N_G(H, K) = \{g \in G \mid gHg^{-1} \subseteq K \} \longrightarrow \mathcal{O}_G(G/H, G/K),
\]
and $\sigma_g = \sigma_h$ for $g, h \in G$ if and only if $Kg = Kh$. For more details, see \cite[Section 2]{Webb}.

Let $A$ be a commutative ring equipped with the discrete topology. Then $A^A$, the set of all maps from $A$ to itself, is equipped with the product topology. Let $\Aut(A)$ be the group of ring automorphisms of $A$, which as a subset of $A^A$ is equipped with the subspace topology. It is easy to check that with respect to this topology $\Aut(A)$ becomes a topological group.

Let $\rho: G \to \Aut(A)$ be a continuous group homomorphism; that is, $\rho$ is a group homomorphism such that for every $a \in A$, the stabilizer subgroup
\[
H_a = \{ g \in G \mid g \cdot a = \rho (g)(a) = a \}
\]
is an open subgroup of $G$. In this case we say that $G$ acts \textit{discretely} on $A$ via $\rho$ (or for short $G$ acts discretely on $A$ when $\rho$ is clear from the context), and define the \textit{skew group ring} $AG$ whose multiplication is given as follows:
\[
(ag) (bh) = a (g \cdot b) gh
\]
for $a, b \in A$ and $g, h \in G$. An $AG$-module $V$ is called \textit{discrete} if $V$ is equipped with the discrete topology and the action of $G$ on it is continuous; that is, given $v \in V$, the stabilizer subgroup $H_v = \{ g \in G \mid g \cdot v = v \}$ is an open subgroup of $G$. Denote the category of discrete $AG$-modules by $AG \Mod^{\dis}$. This is an abelian category since submodules and quotient modules of discrete modules are still discrete.

Now we turn to sheaf theory over atomic sites, for which more details can be found in \cite[Section 2]{DLLX}. Let $J_{at}$ be the atomic Grothendieck topology on $\mathcal{O}_G$. Given a commutative ring $A$ on which $G$ acts discretely, we define a functor $\mathcal{A}$ from $\mathcal{O}_G^{\op}$ to the category of commutative rings as follows. For an object $G/H$ in $\mathcal{O}_G$, we set
\[
\mathcal{A}(G/H) = A^H = \{a \in A \mid h \cdot a = a, \, \forall h \in H \}.
\]
Given a morphism $\sigma_g: G/H \to G/K$ represented by $g \in G$ such that $gHg^{-1} \leqslant K$, the corresponded map $\mathcal{A}(\sigma_g): A^K \to A^H$ is defined by sending $a \in A^K$ to $g^{-1} \cdot a$. It easy to check that $g^{-1} \cdot a$ is fixed by every element in $H$, and $\mathcal{A}(\sigma_g)$ is a ring homomorphism. Thus $\mathcal{A}$ is a presheaf of commutative rings over $\mathcal{O}_G$. Furthermore, by \cite[Thereom III.9.1]{MM}, it is actually a sheaf over the site $(\mathcal{O}_G, J_{at})$. This construction is functorial since up to isomorphism it is exactly the following functor
\[
A \mapsto \Hom_G(-, A).
\]
Thus we obtain a functor $\phi$ from the category of commutative rings on which $G$ acts discretely to the category of structure sheaves over the site $(\mathcal{O}_G, J_{at})$.

Conversely, given a structure sheaf $\mathcal{A}$ of commutative rings over the site $(\mathcal{O}_G, J_{at})$, we define a commutative ring
\[
A = \varinjlim_{H \in \mathcal{H}} \mathcal{A}(G/H)
\]
where the colimit $\psi = \varinjlim_{H \in \mathcal{H}}$ is taken over the poset of open subgroups in $\mathcal{H}$ and inclusions. Furthermore, one can define an action of $G$ on $A$ as follows: for each $a \in A$ and $g \in G$, choose a representative $\tilde{a} \in \mathcal{A}(G/H)$ for a certain $H \in \mathcal{H}$. Then $g$ induces a morphism
\[
\sigma_g: G/g^{-1}Hg \to G/H
\]
in $\mathcal{O}_G$ and hence a ring homomorphism
\[
\mathcal{A}(\sigma_g): \mathcal{A}(G/H) \to \mathcal{A}(G/g^{-1}Hg).
\]
We define $g \cdot a$ to be the equivalence class in $A$ represented by $\mathcal{A}(\sigma_g) (\tilde{a}) \in \mathcal{A}(G/g^{-1}Hg)$.

The proof of \cite[Theorem III.9.1]{MM} shows that the above action is well defined and discrete. We check that it also respects the ring structure of $A$.

Firstly, elements in $G$ act on $A$ as ring automorphisms. Indeed, given $a, b \in A$, one can choose representatives $\tilde{a} \in \mathcal{A}(G/H)$ and $\tilde{b} \in \mathcal{A}(G/K)$. Since $\mathcal{H}$ is a cofinal system of open subgroups, we can find a certain open subgroup $L \in \mathcal{H}$ such that $L \leqslant H \cap K$, and ring homomorphisms
\[
\mathcal{A}(G/H) \to \mathcal{A}(G/L), \quad \mathcal{A}(G/K) \to \mathcal{A}(G/L)
\]
induced by the natural surjections
\[
G/L \to G/H, \quad G/L \to G/K.
\]
Thus without loss of generality one can assume that both $\tilde{a}$ and $\tilde{b}$ are contained in $\mathcal{A}(G/L)$. Given $g \in G$, let $\sigma_g: G/g^{-1}Lg \to G/L$ be the induced morphism in $\mathcal{O}_G$. Then
\[
\mathcal{A} (\sigma_g) (\tilde{a} \tilde{b}) = \mathcal{A} (\sigma_g) (\tilde{a}) \mathcal{A} (\sigma_g) (\tilde{b})
\]
since $\mathcal{A} (\sigma_g)$ is a ring homomorphism. By taking equivalence classes, one has $g \cdot (ab) = (g \cdot a) (g \cdot b)$, so $g$ acts as a ring automorphism on $A$ as desired.

Secondly, given a natural transformation $\alpha: \mathcal{A} \to \mathcal{A}'$, one obtains a ring homomorphism $\psi(\alpha): A \to A'$ which is compatible with the action of $G$; that is, $\psi(\alpha) (g \cdot a) = g \cdot (\psi(\alpha)(a))$ for $g \in G$ and $a \in A$. Thus $\psi$ is a functor from the category of structure sheaves of commutative rings over the site $(\mathcal{O}_G, J_{at})$ to the category of commutative rings on which $G$ acts discretely.

In conclusion, we have the following result.

\begin{proposition} \label{G-rings and structure sheaves}
The functors $\phi$ and $\psi$ give an equivalence between the category of structure sheaves of commutative rings over the site $(\mathcal{O}_G, J_{at})$ and the category of commutative rings on which $G$ acts discretely.
\end{proposition}

\begin{proof}
The conclusion follows from \cite[Theorem III.9.1]{MM} and the above argument.
\end{proof}

By modifying the above proof, we obtain the following result.

\begin{theorem} \label{equivalence of categories}
Let $A$ be a commutative ring on which a topological group $G$ acts discretely, $\mathcal{O}_G$ the orbit category of $G$ with respect to a cofinal system $\mathcal{H}$ of open subgroups, and $\mathcal{A}$ the corresponded structure sheaf over the site $(\mathcal{O}_G, J_{at})$. Then the functor $\phi$ and $\psi$ give an equivalence between $AG \Mod^{\dis}$, the category of discrete $AG$-modules, and $\Sh(\mathcal{O}_G, J_{at}, \mathcal{A})$, the category of sheaves of modules over the ringed site $(\mathcal{O}_G, J_{at}, \mathcal{A})$.
\end{theorem}

In the rest of this paper we call objects in $\Sh(\mathcal{O}_G, J_{at}, \mathcal{A})$ by $\mathcal{A}$-modules.

\section{Finite generation and Noetherianity}

\subsection{Finite generation of rings}

Let $A$ and $\mathcal{A}$ be as specified in Theorem \ref{equivalence of categories}. We say that $\mathcal{A}$ is \textit{finitely generated} if there is a finite set
\[
\{a_1, \ldots, a_n\} \subseteq \bigsqcup_{H \in \mathcal{H}} \mathcal{A}(G/H)
\]
such that every subsheaf (of commutative rings) of $\mathcal{A}$ containing this set coincides with $\mathcal{A}$. We way that $A$ is \textit{finitely generated up to $G$-action} if there is a finite set $\{a_1, \ldots, a_n\} \subseteq A$ such that every subring containing
\[
\{ g \cdot a_i \mid g \in G, i \in [n] \}
\]
coincides with $A$; in other words, $A$ is generated by elements lying in finitely many $G$-orbits.

\begin{lemma} \label{finite generation of rings}
The following statements are equivalent:
\begin{enumerate}
\item $\mathcal{A}$ is finitely generated;

\item $A$ is finitely generated up to $G$-action.
\end{enumerate}
\end{lemma}

\begin{proof}
Suppose that $\mathcal{A}$ is finitely generated. By Proposition \ref{G-rings and structure sheaves}, one can assume that
\[
A = \varinjlim_{H \in \mathcal{H}} \mathcal{A}(G/H),
\]
whose elements are equivalence classes. Given an equivalence class $[a] \in A$, one can find a representative $a \in \mathcal{A}(G/H)$ for a certain $H \in \mathcal{H}$. Since $\mathcal{A}$ is finitely generated, we can find a finite set $\{a_1, \, \ldots, \, a_n\}$ with $a_i \in \mathcal{A}(G/H_i)$ and $H_i \in \mathcal{H}$ which generates $\mathcal{A}$. In particular, $a$ can be written as a polynomial of elements $\mathcal{A} (\sigma_{i,j}) (a_i) \in \mathcal{A}(G/H)$ for finitely many $\sigma_{i, j} \in \mathcal{O}_G(G/H, \, G/H_i)$. But for each $\sigma_{i, j}$ one can find an element $g_{i, j} \in G$ such that $g_{i, j} \cdot [a_i] = [\mathcal{A} (\sigma_{i, j}) (a_i)]$. Consequently, $[a]$ can be written as a polynomial of elements $g_{i, j} \cdot [a_i] \in A$. Thus $A$ is generated by elements lying in the orbits containing $[a_i]$, $1 \leqslant i \leqslant n$.

Conversely, suppose that $A$ is finitely generated up to $G$-action. Then one can find a finite set $\{a_1, \, \ldots, \, a_n\}$ such that $A$ is generated by elements lying in the $G$-orbits $G\cdot a_i$, $1 \leqslant i \leqslant n$. Again by Proposition \ref{G-rings and structure sheaves}, one can assume that $\mathcal{A}(G/H) = A^H$ for $H \in \mathcal{H}$. For each $a_i$, since $\Stab_G(a_i)$ is an open subgroup of $G$, one can find a certain $H_i \in \mathcal{H}$ such that $H_i \leqslant \Stab_G(a_i)$. Then $a_i \in \mathcal{A}(G/H_i)$. Let $\mathcal{B}$ be a subsheaf (of commutative rings) of $\mathcal{A}$ such that $a_i \in \mathcal{B}(G/H_i)$ for $1 \leqslant i \leqslant n$, and let $B = \psi(\mathcal{B})$ which can be regarded as a subring of $A$. Then $a_i \in \mathcal{B} (G/H) = B^H \subseteq B$. Consequently, $\{a_1, \ldots, a_n \} \subseteq B$, which forces $A = B$. Consequently, $\mathcal{B} = \mathcal{A}$, so $\mathcal{A}$ is finitely generated.
\end{proof}

\begin{remark}
The above lemma answers a question raised in \cite[Remark 2.16]{NR1}. That is, the additional condition that $\mathcal{A}$ is a structure sheaf over the site $(\mathcal{O}_G, J_{at})$ implies the converse statement of \cite[Proposition 2.15]{NR1}.
\end{remark}

In the rest of this subsection let $k$ be a commutative Noetherian ring, and $S$ a nonempty set on which $G$ acts discretely. Then $G$ acts on $k[S]$ discretely as well. Indeed, for a polynomial $P \in k[S]$,  we can find a finite subset $T \subseteq S$ such that each indeterminate appearing in $P$ is of the form $x_t$ for a certain $t \in T$. Consequently, one has
\[
\Stab_G(P) \supseteq \bigcap_{t \in T} \Stab_G(t).
\]
Since each $\Stab_G(t)$ is an open subgroup, so is the finite intersection. Thus $\Stab_G(P)$ is also an open subgroup of $G$, and hence $G$ acts on $k[S]$ discretely. In particular, if $S$ is a transitive $G$-set, then one can find an open subgroup $H \leqslant G$ such that the $G$-set $S$ is isomorphic to the $G$-set $G/H$, so $k[S] \cong k[G/H]$, which can be used to characterize finite generation property of $k$-algebras $A$ on which $G$ acts discretely. By convention, when saying that $G$ acts on a $k$-algebra discretely, we always assume that the action of $G$ on $k$ is trivial.

This following result generalizes \cite[Proposition 2.19]{NR1}.

\begin{proposition}
Let $A$ be a $k$-algebra on which $G$ acts discretely. Then $A$ is finitely generated up to $G$-action if and only if there exists a surjective $k$-algebra homomorphism
\[
k[G/H_1] \otimes_k k[G/H_2] \otimes_k \ldots \otimes_k k[G/H_n] \longrightarrow A
\]
with $H_i \in \mathcal{H}$.
\end{proposition}

\begin{proof}
Each $k[G/H_i]$ is generated by one element up to $G$-action, namely the indeterminate corresponded to $H \in G/H$, so the left side is finitely generated up to $G$-action. It is easy to check that its quotient algebras are also finitely generated up to $G$-action.

Conversely, if $A$ is finitely generated up to $G$-action, then one can find elements $a_i \in A$, $1 \leqslant i \leqslant n$, such that $A$ is generated by elements in the $G$-orbits $G \cdot a_i$. For each $a_i$, let $H_i$ be a member in $\mathcal{H}$ which is also a subgroup of $\Stab_G(a_i)$. Then one can construct a $k$-algebra homomorphism $k[G/H_i] \to A$ by sending the indeterminate corresponded to $H_i \in G/H_i$ to $a_i$. This map is well defined by the definition of $H_i$. Since the image of the map specified in the proposition contains each $a_i$, it coincides with $A$. Thus the map is surjective.
\end{proof}

We use a concrete example to illustrate the above construction.

\begin{example} \label{structure sheaves over FI}
Let $S$ be an infinite set equipped with the discrete topology, and let $G$ be the topological group $\Sym(S)$ equipped with the compact-open topology. Given a finite set $T \subseteq S$, the subgroup
\[
H_T = \{ g \in G \mid g \cdot t = t, \, \forall t \in T \}
\]
is an open subgroup of $G$. Furthermore, the family
\[
\mathcal{H} = \{ H_T \leqslant G \mid T \subseteq S, \, |T| < \infty \}
\]
is a cofinal system of open subgroups of $G$; see \cite[III.20, Theorem 3]{Hu}. The orbit category $\mathcal{O}_G$ with respect to $\mathcal{H}$ is isomorphic to the opposite category of the category $\FI_S$.

Now fix a finite subset $T \subseteq S$. The left coset $G/H_T$ is isomorphic to the set $\Inj(T, S)$ consisting of injections from $T$ to $S$, where the action of $G$ on $\Inj(T, S)$ is induced by the action of $G$ on $S$. Indeed, both $G/H_T$ and $\Inj(T, S)$ are transitive $G$-sets. Furthermore, $H_T$ is the stabilizer subgroup of $H_T \in G/H_T$ and the natural inclusion from $T$ to $S$ in $\Inj(T, S)$, so the claim follows.

Let $A = k[G/H_T]$, a polynomial ring whose indeterminates, by the above observation, are parameterized by finite sequences $(s_1, \ldots, s_n)$ of elements in $S$ such that $n = |T|$ and $s_i \neq s_j$ when $i \neq j$. Let $\mathcal{A}$ be the structure sheaf corresponded to $A$. For a finite subset $L$ of $S$, one has $\mathcal{A}(G/H_L) = A^{H_L}$. Clearly $k \subseteq A^{H_L}$. We claim that a non-constant polynomial $P \in A$ is contained in $A^{H_L}$ if and only if the following is true: an indeterminate parameterized by the sequence $(s_1, \ldots, s_n)$ appears in $P$ if and only if every $s_i$ is contained in $L$. The if direction is clearly true. For the other direction, suppose that $P$ has an indeterminate corresponded to a finite sequence $(s_1, \ldots, s_n)$ such that there is a certain $s_i$ not contained in $L$. Then for every $s \in S \setminus L$, one can find a certain $g \in H_L$ such that $g \cdot s_i = s$, so we can obtain infinitely many sequences after applying elements in $H_L$ to $(s_1, \ldots, s_n)$. As a consequence, we deduce that $P$ shall contain infinitely many indeterminates since $P$ is fixed by all elements in $H_L$, which is absurd. Consequently, $\mathcal{A}(G/H_L)$ is the polynomial ring whose indeterminates are indexed by injections from $T$ to $L$ (of course, if $|L| < |T|$, then $\mathcal{A}(G/H_L) = k$).

The reader can check that $\mathcal{A}$ is the same as the structure sheaf $\mathbf{X}^{\FI, d}$ defined in \cite[Definition 2.17]{NR1}, where $d$ is the cardinality of $T$.
\end{example}

\subsection{Finite generation of modules}

Let $V$ be a discrete $AG$-module and $\mathcal{V}$ the corresponded $\mathcal{A}$-module. For $H \in \mathcal{H}$, one has $\mathcal{V}(G/H) = V^H$, the subset of $V$ consisting of elements fixed by all elements in $H$. As we did in the previous subsection, one can define finite generation property of $\mathcal{V}$ and finite generation property of $V$ (viewed as an $A$-module) up to $G$-action.

\begin{lemma} \label{equivalence of finitely generated modules}
The following statements are equivalent:
\begin{enumerate}
\item $V$ is a finitely generated $AG$-module;

\item $V$ viewed as an $A$-module is finitely generated up to $G$-action;

\item $\mathcal{V}$ is a finitely generated $\mathcal{A}$-module.
\end{enumerate}
\end{lemma}

\begin{proof}
The equivalence between (1) and (3) can be established as Lemma \ref{finite generation of rings}, based on Theorem \ref{equivalence of categories}, so we only need to show the equivalence between (1) and (2). Note that $V$ is a finitely generated $AG$-module if and only if there exists a finite set $\{ v_1, \ldots, v_n \}\subseteq V$ such that each $v \in V$ can be written as a finite linear combination $\lambda_1 v_1 + \ldots + \lambda_n v_n$ with $\lambda_i \in AG$. But each $\lambda_i$ can be expressed as a finite linear combination $a_{i1}g_1 + \ldots + a_{is_i} g_{s_i}$. Consequently, statement (1) holds if and only if each $v \in V$ can be written as a finite linear combination of elements in $G \cdot v_1 \cup \ldots \cup G \cdot v_n$ with coefficients in $A$, which is precisely statement (3).
\end{proof}

\begin{remark}
The extra condition that $\mathcal{V}$ is a sheaf of modules rather than a presheaf of modules implies the converse statement of \cite[Proposition 3.14]{NR1}; see \cite[Remark 3.15]{NR1}.
\end{remark}

Given $H \in \mathcal{H}$, we define an $AG$-module $A(G/H)$ with the left coset $G/H$ as a basis. This is a discrete $AG$-module. Indeed, an element $v \in A(G/H)$ can be written as $a_1 g_1H + \ldots + a_n g_n H$, so
\[
\Stab_G(v) \supseteq \Stab_G(a_1) \cap \ldots \cap \Stab_G(a_n) \cap \Stab_G(g_1H) \cap \ldots \cap \Stab_G(g_n H).
\]
Note that each $\Stab_G(a_i)$ is an open subgroup of $G$ and $\Stab_G(g_iH) = g_iHg_i^{-1}$ is also open, so the subgroup $\Stab_G(v)$ contains an open subgroup and hence is open as well.

These special $AG$-modules can be used to test finite generation property of discrete $AG$-modules. The following result generalizes \cite[Proposition 3.18]{NR1}.

\begin{proposition} \label{generators of modules}
The set $\{ A(G/H) \mid H \in \mathcal{H} \}$ is a set of generators for $AG \Mod^{\dis}$. Consequently, a discrete $AG$-module $V$ is finitely generated if and only if there is an epimorphism
\[
\bigoplus_{i=1}^n A(G/H_i) \to V
\]
with $H_i \in \mathcal{H}$.
\end{proposition}

\begin{proof}
Let $V$ be a discrete $AG$-module. For $v \in V$, take $H_v$ in $\mathcal{H}$ such that $H_v \leqslant \Stab_G(v)$. Since there is a homomorphism $A(G/H_v) \to V$ sending $H_v \in G/H_v$ to $v$, we can define a surjection
\[
\bigoplus_{v \in V} A(G/H_v) \longrightarrow V,
\]
establishing the first statement. The second statement follows immediately.
\end{proof}

We give a concrete example to illustrate the above construction.

\begin{example} \label{coincidence}
Let $S$ be an infinite set, $G = \Sym(S)$, and $A$ a commutative Noetherian ring on which $G$ acts discretely. Given a finite subset $T$ of $S$, it is easy to see that $V = A(G/H_T)$ is isomorphic to the free $A$-module with $\Inj(T, S)$ as a basis.

Let $\mathcal{V}$ be the corresponded $\mathcal{A}$-module. For a finite subset $L$ of $S$, $\mathcal{V}(G/H_L) = V^{H_L}$ clearly contains the free $A^{H_L}$-module with $\Inj(T, L)$ as a basis. We claim that $V^{H_L}$ is precisely this free $A^{H_L}$-module. Indeed, let $v = a_1 \sigma_1 + \ldots + a_n \sigma_n$ be an element in $V^{H_L}$ where $a_i \in A$ and $\sigma_i \in \Inj(T, S)$. Using the same argument as in Example \ref{structure sheaves over FI}, we deduce that the image of each $\sigma_i$ is contained in $L$, so it can be viewed as an element in $\Inj(T, L)$. Now for every $g \in H_L$, one has
\[
g \cdot v = (g \cdot a_1) (g \circ \sigma_1) + \ldots + (g \cdot a_n) (g \circ \alpha_n) = (g \cdot a_1) \sigma_1 + \ldots + (g \cdot a_n) \sigma_n = v,
\]
which forces $g \cdot a_i = a_i$. Thus $a_i \in A^{H_L}$.

When $A = k[S]$, the reader can check that $\mathcal{V}$ coincides with the free $\FI$-module $\mathbf{F}^{\FI, d}$ in \cite[Definition 3.16]{NR1}, where $d$ is the cardinality of $T$.
\end{example}

\subsection{Noetherianity}

We say that $A$ is \textit{Noetherian up to $G$-action} if every $G$-invariant ideal of $A$ is finitely generated up to $G$-action. An $\mathcal{A}$-module $\mathcal{V}$ is Noetherian if every $\mathcal{A}$-submodule of $\mathcal{V}$ is finitely generated.

The following proposition generalizes \cite[Theorem 4.6]{NR1} and implies its converse for $G = \Sym(S)$.

\begin{proposition} \label{equivalence of Noetherianity}
Let $V$ be a discrete $AG$-module. Then one has:
\begin{enumerate}
\item $V$ is Noetherian if and only if the corresponded $\mathcal{A}$-module $\mathcal{V}$ is Noetherian;

\item $A$ is a Noetherian up to $G$-action if and only if viewed as an $AG$-module it is Noetherian;

\item every finitely generated discrete $AG$-module is Noetherian if and only if $A(G/H)$ is a Noetherian $AG$-module for $H \in \mathcal{H}$.
\end{enumerate}
\end{proposition}

\begin{proof}
The first statement follows from Theorem \ref{equivalence of categories} and Lemma \ref{equivalence of finitely generated modules}. The second statement is also clear since $G$-invariant ideals of $A$ are precisely discrete $AG$-submodules of $A$. The third statement follows from Proposition \ref{generators of modules}.
\end{proof}

\section{Actions of highly transitive groups on polynomial rings}

Throughout this section let $S$ be an infinite set, $k$ a commutative Noetherian ring, $A$ a commutative ring, and $G$ a group acting on $S$ via a group homomorphism $\rho: G \to \Sym(S)$. Fix $\mathcal{H}$ to be the cofinal system consisting of pointwise stabilizer subgroups
\[
H_T = \{ g \in G \mid g \cdot t = t, \, \forall t \in T \}
\]
with $T$ is a finite subset of $S$. With respect to this cofinal system $G$ becomes a topological group by \cite[III.20, Theorem 3]{Hu}, and it acts on $S$ discretely. Let $\mathcal{O}_G$ be the orbit category of $G$ with respect to the cofinal system $\mathcal{H}$.

\begin{definition}
Let $L$ and $T$ be two finite subsets of $S$. An injective map $\sigma: L \to T$ is called extendable if there is an element $\tilde{\sigma} \in G$ such that $\sigma(l) = \tilde{\sigma}(l)$ for every $l \in L$.
\end{definition}

It is clear that the composite of two extendable injections is also an extendable injection.

\begin{definition} \cite[Definition 3.11]{BMMN}
We say that the action of $G$ on $S$ is \textit{$n$-transitive} if for a fixed $n \in \mathbb{N}$ and any two entrywise distinct sequences
\[
(s_1, \, s_2, \, \ldots, \, s_n), \quad (t_1, \, t_2, \, \ldots, \, t_n),
\]
there exists $\sigma \in G$ such that $\sigma(s_i) = t_i$ for $i \in [n]$. The action is said to be \textit{highly transitive} if it is $n$-transitive for every $n \in \mathbb{N}$.
\end{definition}

Since the set of entrywise distinct sequences of length $n$ can be identified with the set $\Inj([n], S)$, the action of $G$ on $S$ is highly transitive if and only if $G$ acts transitively on $\Inj([n], S)$ for $n \in \mathbb{N}$.

\begin{lemma} \label{extendable maps}
The action of $G$ on $S$ is highly transitive if and only if injections between two finite subsets of $S$ are extendable. In this case, the map sending finite subsets $T$ of $S$ to $G/H_T$ is injective.
\end{lemma}

\begin{proof}
The if direction is clear since given two entrywise distinct sequences $(s_1, \, s_2, \, \ldots, \, s_n)$ and $(t_1, \, t_2, \, \ldots, \, t_n)$, the injection sending $s_i$ to $t_i$ can be extended to a certain $\sigma \in G$. For the other direction, note that an injection $\sigma: L \to T$ can be identified with an entrywise distinct sequence
\[
(\sigma(l_1), \, \sigma(l_2), \, \ldots, \, \sigma(l_m), \, t_{m+1}, \, \ldots, \, t_n)
\]
whose entries exhaust all elements in $T$. Then any element $\tilde{\sigma} \in G$ sending $(l_1, \, \ldots, \, l_m, \, t_{m+1}, \, \ldots, \, t_n)$ to the above sequence extends $\sigma$.

Now we prove the second statement. If the map is not injective, then there are finite subsets $K \neq L$ such that $H_K = H_L$, so one has $K \neq K \cup L$ but $H_K = H_{K \cup L}$ or $L \neq K \cup L$ but $H_L = H_{K \cup L}$. Consequently, one gets a finite set $T$ and a proper subset $T' \subsetneqq T$ such that $H_T = H_{T'}$. Write $T$ as a sequence
\[
\mathbf{t} = (t_1, \, \ldots, \, t_m, \, t_{m+1}, \, \ldots, \, t_{m+n})
\]
such that the first $m$ elements form $T'$, and define another sequence $\mathbf{t}'$ via replacing $t_{m+n}$ by another element not in $T$. Then there is no $\sigma \in G$ sending the sequence $\mathbf{t}$ to $\mathbf{t}'$ since if $\sigma$ fixes the first $m$ terms, it must fix the whole sequence. Therefore, the action of $G$ on $S$ is not highly transitive.
\end{proof}

\begin{proposition} \label{characterization for FI}
The following statements are equivalent:
\begin{enumerate}
\item there is an isomorphism $\varrho: \FI^{\op}_S \to \mathcal{O}_G$ such that $\varrho(T) = G/H_T$ for every finite $T \subseteq S$;

\item the action of $G$ on $S$ is highly transitive;

\item the image of $\rho: G \to \Sym(S)$ is a dense subgroup of $\Sym(S)$.
\end{enumerate}
\end{proposition}

\begin{proof}
$(1) \Rightarrow (2)$. Given two entrywise distinct sequences $\mathbf{t} = (t_1, \, \ldots, \, t_n)$ and $\mathbf{t}' = (t_1', \, \ldots, \, t_n')$, and let $T$ and $T'$ be the corresponded finite subsets of $S$. Since $T$ and $T'$ are isomorphic in $\FI^{\op}_S$, so must be $G/H_T$ and $G/H_{T'}$ in $\mathcal{O}_G$. In particular, there is $g \in G$ such that $gH_Tg^{-1} = H_{T'}$, so $g(t) = \sigma g(t)$ for all $t \in T$ and $\sigma \in H_{T'}$. It follows that $g(t) \in T'$ since otherwise every element in $H_{T'}$ fixes $T'' = T' \sqcup \{g(t) \}$ and hence $H_{T'} = H_{T''}$, contradicting the assumption that the map $T \mapsto G/H_T$ is an bijection.

We have deduced that every $g \in G$ representing an isomorphism from $G/H_T$ to $G/H_{T'}$ gives a bijection from $T$ to $T'$ via restricting $g$ to $T'$. Furthermore, by the definition of orbit categories we know that $g$ and $g'$ in $G$ representing the same isomorphism from $G/H_T$ to $G/H_{T'}$ if and only if $H_{T'} g = H_{T'} g'$, and if and only if their restrictions coincide as maps from $T$ to $T'$. Since $\varrho$ is an isomorphism, the number of morphisms from $G/H_T$ to $G/H_{T'}$ equals the number of bijections from $T'$ to $T$. It follows from these observations that there is a certain $g \in G$ representing a morphism from $G/H_T$ to $G/H_{T'}$ whose restriction to $T'$ sends $\mathbf{t}'$ to $\mathbf{t}$; that is, the action of $G$ on $S$ is highly transitive.

$(2) \Rightarrow (1)$. The proof is similar to that of \cite[Theorem III.9.2]{MM}. For the convenience of the reader, we describe an explicit construction of $\varrho: \FI_S^{\op} \to \mathcal{O}_G$ which is fully faithful and is a bijection between the object sets. Since the map sending a finite subset $T \subseteq S$ to $G/H_T$ is the desired bijection by Lemma \ref{extendable maps}, we only need to define $\varrho$ for morphisms.

Given a morphism $\varsigma: T \to L$ in $\FI_S^{\op}$, we obtain a unique injection $\sigma: L \to T$. By Lemma \ref{extendable maps}, $\sigma$ extends to an element $\tilde{\sigma} \in G$. Since $\tilde{\sigma}^{-1}H_T \tilde{\sigma}$ is contained in $H_L$, $\tilde{\sigma}^{-1}$ represents a morphism from $G/H_T$ to $G/H_L$ in $\mathcal{O}_G$ sending $gH_T$ to $g \tilde{\sigma} H_L$, so we can define $\varrho(\varsigma)$ to be this morphism.

We check that the above construction is well defined. Indeed, if $\tilde{\sigma}$ and $\tilde{\tau}$ are extensions of $\sigma: L \to T$ and $\tau: L \to T$ respectively, then for any $g \in G$,
\begin{align*}
& g \tilde{\sigma} H_L = g \tilde{\tau} H_L \quad \Longleftrightarrow \quad \tilde{\sigma} H_L = \tilde{\tau} H_L \quad \Longleftrightarrow \quad \tilde{\tau}^{-1} \tilde{\sigma} H_L = H_L\\
& \Longleftrightarrow \quad \tilde{\tau}^{-1} \tilde{\sigma} \in H_L \quad \Longleftrightarrow \quad \tilde{\tau}^{-1} \tilde{\sigma}(x) = x, \, \forall x \in L\\
& \Longleftrightarrow \tilde{\tau}^{-1} \sigma (x) = x, \, \forall x \in L \quad \Longleftrightarrow \quad \sigma(x) = \tilde{\tau} (x) = \tau(x), \, \forall x \in L.
\end{align*}
Thus this construction is independent of the choice of extensions, and hence is well defined. It also shows that the construction is faithful. It is routine to check that $\varrho$ is indeed a functor, and it is full since every morphism in $\FI_S$ extends to an element in $G$ and every morphism in $\mathcal{O}_G$ is induced by an element in $G$.

$(2) \Leftrightarrow (3)$. This is well known to experts; see for instance \cite[Section 4.2]{Ca}. For the convenience of the reader we give a detailed proof. Without loss of generality we can assume that $\rho$ is an injective homomorphism, and in this case the image of $\rho$ is precisely $G$. Note that the collection $\{ \sigma U_T \}$, where $\sigma \in \Sym(S)$ and $U_T$ is the stabilizer subgroup of $\Sym(S)$ fixing every element in a finite subset $T \subseteq S$, is a base of the topology on $\Sym(S)$. Therefore, $G$ is dense in $\Sym(S)$ if and only if $G \cap \sigma U_T$ is nonempty for every finite subset $T \subseteq S$ and every $\sigma \in \Sym(S)$; that is, there is a certain $g \in G$ such that $g \in \sigma U_T$. But this is equivalent to saying that for every $\sigma \in \Sym(S)$ and every finite subset $T \subseteq S$, there is a certain $g \in G$ such that $g$ and $\sigma$ induce the same map from $T$ to $\sigma(T)$. But the action of $\Sym(S)$ on $S$ is highly transitive, so this is true if and only if the action of $G$ on $S$ is highly transitive.
\end{proof}

In the rest of this section suppose that the action of $G$ on $S$ is highly transitive, so one may identify $\mathcal{O}_G$ with $\FI_S^{\op}$. We immediately have the following result.

\begin{theorem} \label{main result 1'}
Let $A$ be a commutative ring on which $\tilde{G} = \Sym(S)$ and $G$ act discretely. If the action of $G$ on $S$ is highly transitive, then
\[
A \tilde{G} \Mod^{\dis} \simeq AG \Mod^{\dis}.
\]
\end{theorem}

\begin{proof}
The conclusion follows from Theorem \ref{equivalence of categories} and Proposition \ref{characterization for FI}.
\end{proof}

The following theorem has been established for $G = \Sym(S)$; see for instances \cite{AH, Co, HS}.

\begin{theorem} \label{Noetherian wrt permutation groups}
If the action of $G$ on $S$ is highly transitive, then every finitely generated discrete $k[S]G$-module is Noetherian. In particular, $k[S]$ is a Noetherian $k[S]G$-module.
\end{theorem}

\begin{proof}
By Proposition \ref{characterization for FI}, the orbit category $\mathcal{O}_G$ is isomorphic to the opposite category of $\FI_S$. Let $\mathcal{A}$ be the structure sheaf over the site $(\mathcal{O}_G, J_{at})$ corresponded to $A = k[S]$. By Example \ref{structure sheaves over FI}, for finite subsets $\mathcal{T}$ of $S$,
\[
\mathcal{A}(T) = k[S]^{H_T} = k[T];
\]
and for an injection $\sigma: L \to T$, $\mathcal{A}(\sigma)$ is the inclusion
\[
k[L] \longrightarrow k[T], \quad x_l \mapsto x_{\sigma(l)}.
\]
In other words, $\mathcal{A}$ is the structure sheaf $\mathbf{X}^{\FI, 1}$ in \cite[Definition 2.17]{NR1}.

Now let $V$ be a finitely generated discrete $AG$-module. By Proposition \ref{equivalence of Noetherianity}, it suffices to show the Noetherianity of the corresponded $\mathcal{A}$-module $\mathcal{V}$. But by \cite[Theorem 6.15]{NR1}, $\mathcal{V}$ as a presheaf of modules is Noetherian, which automatically implies the Noetherianity of $\mathcal{V}$ as an $\mathcal{A}$-module. Indeed, given a $\mathcal{A}$-submodule $\mathcal{W}$ of $\mathcal{V}$, we can viewed $\mathcal{W}$ as a subpresheaf of $\mathcal{V}$. Since $\mathcal{V}$ is Noetherian as a presheaf of modules, $\mathcal{W}$ as a presheaf is finitely generated. Therefore, by \cite[Proposition 3.18]{NR1}, we obtain a surjection
\[
\bigoplus_{i=1}^n \mathbf{F}^{\FI, d_i} \longrightarrow \mathcal{W}
\]
of presheaves of modules. But by Example \ref{coincidence} $\mathbf{F}^{\FI, d_i}$ is the $\mathcal{A}$-module corresponded to the discrete $AG$-module $A(G/H_{T_i})$ with $|T_i| = d_i$, so this map is also a surjection of $\mathcal{A}$-modules. It follows from Lemma \ref{equivalence of finitely generated modules} and a sheaf theoretic version of Proposition \ref{generators of modules} that $\mathcal{W}$ is a finitely generated $\mathcal{A}$-module, so $\mathcal{V}$ is Noetherian.
\end{proof}

\begin{remark}
This theorem actually holds for the bigger $k$-algebra $A = k[x_{i,s} \mid i \in [n], \, s \in S]$ where the action of $G$ on $A$ is determined by $\sigma \cdot x_{i, s} = x_{i, \sigma(s)}$ for $\sigma \in G$ and $s \in S$.
\end{remark}

\begin{remark}
Since for any infinite sets $S$ and $S'$, the categories of $\FI_S$ and $\FI_{S'}$ are equivalent, the above theorem holds for any infinite set if and only if it holds for the set $S = \mathbb{N}$. This fact has been observed in \cite[Remark 3.3]{AH} for $G = \Sym(S)$.
\end{remark}

\section{Actions of order-preserving permutations on polynomial rings}

For $S = \mathbb{N}$ and $\mathcal{I}$ the set of strictly increasing function from $\mathbb{N}$ to itself, in \cite{NR1} the authors proved that $k[S]$ is a Noetherian module over the skew monoind ring $A\mathcal{I}$. This motivates us to impose a linear order $\leqslant$ on $S$ and consider the action of the group $\Aut(S, \leqslant)$ of all order-preserving permutations on $k[S]$.

Throughout this section let $(S, \leqslant)$ be an infinite linearly ordered set, $k$ a commutative Noetherian ring, $A$ a commutative ring, and $G$ a group acting on $(S, \leqslant)$ via a group homomorphism $\rho: G \to \Aut(S, \leqslant)$. Fix $\mathcal{H}$ to be the cofinal system consisting of pointwise stabilizer subgroups
\[
H_T = \{ g \in G \mid g \cdot t = t, \, \forall t \in T \}
\]
with $T$ a finite subset of $S$. With respect to this cofinal system $G$ becomes a topological group by \cite[III.20, Theorem 3]{Hu}, and it acts on $S$ discretely. Let $\mathcal{O}_G$ be the orbit category of $G$.

\subsection{Homogeneous linear orders}

Given $a \in S$, denote by $(-, a)$ the subset of $S$ consisting of elements $b$ with $b < a$. Similarly, one defines $(a, -)$. For $a, b \in S$ with $a < b$, let $(a, b)$ be the intersection of $(-, b)$ and $(a, -)$. We call $(-, a)$, $(a, -)$ and $(a, b)$ open intervals of $S$. They are linearly ordered sets equipped with the order inherited from that on $S$.

Now we define a special linear order on $S$, which plays a central role in this section.

\begin{definition} \label{piecewise homogeneous}
A linear order $\leqslant$ on $S$ is called \textit{homogeneous} if for pairs $a < b$ and $a' < b'$ in $S$, one has $(a, b) \cong (a', b')$, $(-, a) \cong (-, a')$, and $(a, -) \cong (a', -)$ as linearly ordered sets. It is called \textit{globally homogeneous} if every open interval of $S$ of form $(-, a)$, $(a, b)$ or $(a, -)$ is isomorphic to $(S, \leqslant)$ as linearly ordered sets.
\end{definition}

A few elementary properties of (globally) homogenous linear orders are collected in the next two lemmas.

\begin{lemma} \label{unbounded and dense}
Let $\leqslant$ be a linear order on $S$.
\begin{enumerate}
\item If $\leqslant$ is homogenous, then it is dense and unbounded.

\item If $\leqslant$ is unbounded, then it is homogenous if and only if for $a < b$ and $a' < b$ in $S$, one has $(a, b) \cong (a', b')$.

\item The linear order $\leqslant$ is globally homogenous if and only if for every $a < b$ in $S$, one has $(a, b) \cong (S, \leqslant)$.
\end{enumerate}
\end{lemma}

\begin{proof}
(1). If $(S, \leqslant)$ has an endpoint, say, a minimal element $a$. Then for any $b \neq a$ in $S$, the open interval $(-, a)$ is empty, while $(-, b)$ is nonempty, so they cannot be isomorphic. Thus $(S, \leqslant)$ is unbounded. Similarly, if $S$ is not dense, then we can find an increasing pair $a < b$ in $S$ with $(a, b) = \emptyset$. Since $S$ is unbounded, we can find another $c \in S$ with $c > b$. Then $(a, b)$ is not isomorphic to $(a, c)$. This contradiction tells us that $(S, \leqslant)$ is dense.

(2). The only if direction is clear. For the other direction, since $\leqslant$ is unbounded, we can find a certain $c \in S$ such that $c > a$ and $c > a'$. Then one has $(a, c) \cong (a', c)$ and hence
\[
(a, -) = (a, c) \sqcup [c, -) \cong (a', c) \sqcup [c, -) = (a', -).
\]
Similarly, one can construct a desired isomorphism $(-, a) \cong (-, a')$.

(3). The only if direction is clear. For the if direction, one needs to show $(-, a) \cong (S, \leqslant) \cong (a, -)$ for $a \in S$. We prove the first isomorphism since the second one can be verified similarly.

We claim that $\leqslant$ is unbounded. Otherwise, suppose that $s_0$ is the minimal element in $S$. For any $s < t$ in $S$, let $\sigma: (s, t) \to (S, \leqslant)$ be an isomorphism. Then one can find a certain $s' \in (s, t)$ such that $\sigma(s') = s_0$. Consequently, one has $s' > s$ and $s'$ is the minimal element in $(s, t)$. But then $(s, s')$ is the empty set, which cannot be isomorphic to $(S, \leqslant)$. This contradiction tells us that $S$ has no minimal element. Similarly, one can show that it has no maximal element.

Choose an element $x \in (s, t)$. Then the isomorphism $\sigma: (s, t) \to (S, \leqslant)$ induces an isomorphism $(s, x) \cong (-, \sigma(x))$. But $(s, x) \cong (S, \leqslant)$ as well, so it remains to show that $(-, a) \cong (-, \sigma(x))$. But this is clear since by statement (2), $\leqslant$ is homogenous.
\end{proof}

\begin{example}
The if direction of statement (2) is false without the assumption that $\leqslant$ is unbounded. For example, let $S$ be the set of nonnegative real numbers with the usual linear order. Then $(a, b) \cong (a', b')$ for $a < b$ and $a' < b'$ in $S$, but clearly $(-, 0) = \emptyset \ncong (-, 1)$.

A dense and unbounded linear order might not be homogeneous. For an example, consider the following order on $\mathbb{R}$: every irrational number is strictly smaller than each rational numbers, and irrational numbers (resp., rational numbers) are ordered by the usual order.

Clearly, globally homogeneous linear orders are homogeneous, but the converse statement is false. For an example, let $S = \omega_1 \times \mathbb{Q}$ where $\omega_1$ is the least uncountable well ordered set. Impose the lexicographic order on $S$ which is unbounded. This linear order is homogenous since for every $a < b$ in $S$, the open interval $(a, b)$ is isomorphic to $(\mathbb{Q}, \leqslant)$. However, it cannot be globally homogenous since $(a, b)$ is countable while $S$ is uncountable.
\end{example}

We will characterize homogeneous linear orders via actions of order-preserving permutation groups on $S$. For this purpose, we introduce the following definition.

\begin{definition} \cite[Definition 3.16]{BMMN}
We say that the action of $G$ is \textit{$n$-homogenous} on $(S, \leqslant)$ if for a fixed $n \in \mathbb{N}$ and every pair of increasing sequences $a_1 < a_2 < \ldots < a_n$ and $b_1 < b_2 < \ldots < b_n$ in $S$, there is an element $\sigma \in G$ such that $\sigma(a_i) = b_i$ for $i \in [n]$. The action is said to be \textit{highly homogenous} if it is $n$-homogenous for every $n \in \mathbb{N}$.
\end{definition}

Properties described in the next lemma shall be well known to experts; see for instances \cite{BMMN} or \cite{Ca}. For the convenience of the reader, we give detailed proofs.

\begin{lemma} \label{elementary properties}
The following statements are equivalent:
\begin{enumerate}
\item the linear order $\leqslant$ is homogenous;

\item there is a group whose action on $(S, \leqslant)$ is 2-homogenous;

\item there is a group whose action on $(S, \leqslant)$ is highly homogenous;

\item every order-preserving injections between two finite subsets of $S$ is extendable;

\item every order-preserving bijection between two finite subsets of $S$ is extendable.
\end{enumerate}
\end{lemma}

\begin{proof}
$(1) \Rightarrow (4)$. Let $\sigma: T \to L$ be an order-preserving injection. Without loss of generality we can assume that $T \neq \emptyset$. Order elements in $T$ as $a_1 < \ldots < a_m$ and elements in $L$ as $b_1 < \ldots < b_n$ with $m \leqslant n$. By Lemma \ref{unbounded and dense}, one can choose a certain $s, t \in S$ such that $s < a_1 < a_m < t$ and $s < b_1 < b_n < t$. A desired order-preserving permutation $\tilde{\sigma}$ can be constructed via glueing $\sigma$, the following isomorphisms
\[
(s, a_1) \cong (s, \sigma(a_1)), \, \ldots, \, (a_{m-1}, a_m) \cong (\sigma(a_{m-1}), \sigma(a_m)), \, (a_m, t) \cong (\sigma(a_m), t),
\]
and identities on $(-, s)$ and $(t, -)$.

$(4) \Rightarrow (3)$. For $n \in \mathbb{N}$, $a_1 < \ldots < a_n$, and $b_1 < \ldots < b_n$, let $L$ and $T$ be the subsets of $S$ consisting of $a_i$'s and $b_i$'s respectively. Then there is a unique order-preserving bijection $\sigma: L \to T$ extending to an order-preserving bijection $\tilde{\sigma} \in \Aut(S, \leqslant)$. Consequently, the action of $\Aut(S, \leqslant)$ on $(S, \leqslant)$ is highly homogenous.

$(3) \Rightarrow (2)$ and $(4) \Rightarrow (5)$ are clear.

$(2) \Rightarrow (1)$. Take $a < b$ and $a' < b'$ in $S$ and choose an element $\sigma$ in the group such that $\sigma(a) = a'$ and $\sigma(b) = b'$. Clearly, $\sigma$ induces the desired isomorphisms $(-, a) \cong (-, a')$, $(a, b) \cong (a', b')$, and $(b, -) \cong (b', -)$.

$(5) \Rightarrow (4)$. Given an order-preserving injection $\sigma$ from a finite subset $T$ of $S$ to another finite subset $L$ of $S$, we obtain an order-preserving bijection $\sigma'$ from $T$ to $\sigma(T)$. By (5), $\sigma'$ extends to an order-preserving bijection $\tilde{\sigma}: S \to S$, which clearly extends $\sigma$.
\end{proof}

\begin{remark}
When $\leqslant$ is a homogeneous linear order on $S$, by mimicking the proof of Proposition \ref{characterization for FI}, one can show that a subgroup $G \leqslant \Aut(S, \leqslant)$ satisfies (2) or (3) of the above lemma if and only if it is a dense subgroup of $\Aut(S, \leqslant)$. However, if $\leqslant$ is not homogeneous, then dense subgroups of $\Aut(S, \leqslant)$ might not have these properties. A trivial example is $\Aut(\mathbb{N}, \leqslant)$ with the usual linear order on $\mathbb{N}$.
\end{remark}

\subsection{Noetherian results}

Let $\OI_S$ the category of finite subsets of $S$ and order-preserving injections. We can prove the following result parallel to Proposition \ref{characterization for FI}.

\begin{proposition} \label{characterization of OI}
There is an isomorphism $\varrho: \OI_S^{\op} \to \mathcal{O}_G$ with $\varrho(T) = G/H_T$ for finite subsets $T$ of $S$ if and only if the action of $G$ on $(S, \leqslant)$ is highly homogenous.
\end{proposition}

\begin{proof}
\textbf{The only if direction.} Let $t_1 < \ldots < t_n$ and $t_1' < \ldots < t'_n$ be two increasing sequences in $S$, and let $T$ (resp., $T'$) be the set formed by elements in the first (resp., second) sequences. Since $T$ is isomorphic to $T'$ in $\OI_S^{\op}$, $G/H_T$ and $G/H_{T'}$ are isomorphic in $\mathcal{O}_G$. Thus we can find a certain $\sigma \in G$ such that $\sigma H_T \sigma^{-1} = H_{T'}$. As in the proof of Proposition \ref{characterization for FI}, one can show that $\sigma$ is an extension of the unique increasing map from $T'$ to $T$. Thus the action of $G$ on $(S, \leqslant)$ is highly homogenous.

\textbf{The if direction.} Firstly we show that the map $T \mapsto H_T$ is a bijection. Since it is clearly surjective, it suffices to show the injectivity. Suppose that $T \neq L$ are two finite subsets of $S$ and without loss of generality assume that there is an element $x$ contained in $T$ but not contained in $L$. Order elements in $T \cup L$ as a sequence $x_1 < x_2 < \ldots < x_n$ and suppose that $x = x_i$. Note that the linear order $\leqslant$ is homogeneous by Lemma \ref{elementary properties}, and hence dense by Lemma \ref{unbounded and dense}, so we can replace $x_i$ by another element $y$ not contained in this sequence to obtain a new increasing sequence. Again by Lemma \ref{elementary properties} we can find a certain $\sigma \in G$ sending the first sequence to the second sequence. Clearly, $\sigma$ is contained in $H_L$ but not in $H_T$, establishing the injectivity of the map $T \mapsto H_T$. The desired functor $\varrho$ can be constructed by slightly modifying the proof of Proposition \ref{characterization for FI}.
\end{proof}

\begin{remark}
If $(S, \leqslant)$ is a countable dense linearly ordered set without endpoints, and $G = \Aut(S, \leqslant)$, the fact that $\mathcal{O}_G$ is isomorphic to $\OI_S^{\op}$ is already known; see \cite[Example D.3.4.11]{Jo}.
\end{remark}

\begin{theorem}
Let $A$ be a commutative ring on which $\tilde{G} = \Aut(S, \leqslant)$ and $G$ act discretely. If the action of $G$ on $(S, \leqslant)$ is highly homogenous, then
\[
A \tilde{G} \Mod^{\dis} \simeq AG \Mod^{\dis}.
\]
\end{theorem}

\begin{proof}
The conclusion follows from Theorem \ref{equivalence of categories} and Proposition \ref{characterization of OI}.
\end{proof}

We are ready to prove the following result.

\begin{theorem} \label{Noetherian wrt homogeneous order}
If the action of $G$ on $(S, \leqslant)$ is highly homogenous, then every finitely generated discrete $k[S]G$-module is Noetherian. In particular, $k[S]$ is a Noetherian $k[S]G$-module.
\end{theorem}

\begin{proof}
Applying the same argument as that in the proof of Theorem \ref{Noetherian wrt permutation groups}, one gets the following description of the structure sheaf $\mathcal{A}$ over the site $(\OI_S, J_{at})$ corresponded to the algebra $A = k[S]$: for a finite subset $T$ of $S$,
\[
\mathcal{A}(T) = A^{H_T} = k[T];
\]
for an order-preserving injection $\alpha: L \to T$, $\mathcal{A}(\alpha)$ is the inclusion
\[
k[L] \longrightarrow k[T], \quad x_l \mapsto x_{\alpha(l)}.
\]
In other words, $\mathcal{A}$ coincides with the structure sheaf $\mathbf{X}^{\OI, 1}$ in \cite[Definition 2.17]{NR1}.

Let $V$ be a finitely generated discrete $AG$-module. By Proposition \ref{equivalence of Noetherianity}, it suffices to show that the corresponded $\mathcal{A}$-module $\mathcal{V}$ is Noetherian. But this has been established in \cite[Theorem 6.15]{NR1}. Indeed, the conclusion proved \cite{NR1} is that $\mathcal{V}$ as a presheaf of modules  is Noetherian, automatically implying the Noetherianity of $\mathcal{V}$ as an $\mathcal{A}$-module.
\end{proof}

\begin{remark}
This theorem holds for the bigger $k$-algebra $A = k[x_{i,s} \mid i \in [n], \, s \in S]$ where the action of $G$ on $A$ is determined by $\sigma \cdot x_{i, s} = x_{i, \sigma(s)}$ for $\sigma \in G$ and $s \in S$.
\end{remark}

In the rest of this subsection we give a few applications of Theorem \ref{Noetherian wrt homogeneous order}.

\begin{corollary}
Let $\leqslant$ be a homogenous linear order on $S$, and $\mathcal{I}$ the monoid of order-preserving injections on $(S, \leqslant)$. Then $k[S]$ is a Noetherian module over the skew monoid ring $k[S]\mathcal{I}$.
\end{corollary}

\begin{proof}
Let $V$ be a $k[S]\mathcal{I}$-submodule of $k[S]$ and $G = \Aut(S, \leqslant)$. Then $V$ viewed as a $k[S]G$-submodule of $k[S]$ shall be finitely generated by Theorem \ref{Noetherian wrt homogeneous order}. Since $G$ is a submonoid of $\mathcal{I}$, it follows that $V$ as a $k[S] \mathcal{I}$-module is finitely generated as well. Therefore, $k[S]$ is a Noetherian $k[S]\mathcal{I}$-module.
\end{proof}

\begin{remark}
This corollary does not imply statement (i) of \cite[Corollary 6.21]{NR1} since the linear order on $\mathbb{N}$ is not homogenous.
\end{remark}

Note that a homogenous linear order $\leqslant$ on $S$ determines a topology $\mathcal{T}_{\mathrm{ord}}$ on $S$ with a basis
\[
\mathcal{B} = \{ (a, b) \mid a < b, a, b \in S \}.
\]

\begin{corollary}
Let $G = \Homeo(S, \mathcal{T}_{\mathrm{ord}})$ be the group of self-homeomorphisms. Then $k[S]$ is a Noetherian $k[S]G$-module.
\end{corollary}

\begin{proof}
The proof is based on a very transparent observation: if $H \leqslant K$, then $K$-invariant ideals of $k[S]$ form a subposet of the poset of $H$-invariant ideals of $k[S]$. Consequently, if $k[S]$ is Noetherian up to the action of $H$, it is also Noetherian up to the action of $K$. Therefore, it is enough to show that $\Aut(S, \leqslant)$ is a subgroup of $G$, namely every order-preserving bijection is a self-homeomorphism on $(S, \mathcal{T}_{\mathrm{ord}})$. But this is trivially true.
\end{proof}

As an important consequence of Theorem \ref{Noetherian wrt homogeneous order}, we show that $k[S]$ is a Noetherian discrete $k[S]G$-module whenever the action of $G$ on $S$ is highly homogenous. For this purpose, we introduce the following fundamental result of Cameron; see \cite[Theorem 6.1]{Ca76}, \cite[3.10]{Ca}, or \cite[Theorem 13.1]{BMMN}.

\begin{theorem} \label{classification}
Suppose that the action of $G \leqslant \Sym(S)$ on $S$ is highly homogenous. Then either $G$ is dense subgroup of $\Sym(S)$, or a dense subgroup of the group of permutations preserving one of the following relations on $S$:
\begin{enumerate}
\item a dense linear order $\leqslant$,

\item a dense linear betweenness relation $\mathcal{B}$,

\item a dense cyclic order $\mathcal{K}$,

\item and a dense separation relation $\mathcal{S}$.
\end{enumerate}
\end{theorem}

The word ``dense" appearing at different positions of the above theorem have different meanings. The first two are of a topological sense: we impose the discrete topology on $S$ and the compact-open topology on $\Sym(S)$. The other ones have the usual meaning in order theory.

\begin{corollary} \label{highly homogenous result}
If the action of $G$ on $S$ via a group homomorphism $\rho: G \to \Sym(S)$ is highly homogenous, then $k[S]$ is a Noetherian $k[S]G$-module.
\end{corollary}

\begin{proof}
Without loss of generality we can assume that $G$ is a subgroup of $\Sym(S)$. We have established the conclusion for dense subgroups of $\Sym(S)$ and $\Aut(S, \leqslant)$ in Theorem \ref{main result 1'} and Theorem \ref{Noetherian wrt homogeneous order}. Thus it remains to check the conclusion for dense subgroups of $\Aut(S, \mathcal{B})$, $\Aut(S, \mathcal{K})$, or $\Aut(S, \mathcal{S})$. But as we did before, we can show the following fact: If $G$ is a dense subgroup of $\Aut(S, \mathcal{B})$, then its orbit category is equivalent to the category of finite sets of $S$ equipped with linear betweenness relations  and structure-preserving injections. Therefore, by the downward L\"{o}wenheim–Skolem theorem (see \cite[Section 2.3]{Marker}), Theorem \ref{equivalence of categories}, and the well known fact that this structure is $\aleph_0$-categorical (see \cite[Section 2.5, (2.10)]{Ca}), the conclusion is independent of particular choices of sets $S$, linear betweenness relations on $S$, and dense subgroups $G \leqslant \Aut(S, \mathcal{B})$. Similarly, one can prove the independence result for $\mathcal{K}$ or $\mathcal{S}$. Consequently, we only need to check the conclusion for the following three groups:
\begin{itemize}
\item $G_1 = \Aut(\mathbb{Q}, \mathcal{B})$,
\item $G_2 = \Aut(\mathbb{Q}, \mathcal{K})$,
\item $G_3 = \Aut(\mathbb{Q}, \mathcal{S})$,
\end{itemize}
where $\mathcal{B}$, $\mathcal{K}$ and $\mathcal{S}$ are induced from the natural linear order on $\mathbb{Q}$. For details, see \cite[Section 11.3]{BMMN}.

By \cite[Theorems 11.7]{BMMN}, $G_1$ contains $\Aut(\mathbb{Q}, \leqslant)$ as a normal subgroup of index 2, so the conclusion holds for $G_1$. By \cite[Theorem 11.11]{BMMN}, $G_3$ contains $G_2$ as a normal subgroup of index 2, so it is enough to prove the conclusion for $G_2$. But by \cite[Theorem 11.9]{BMMN}, $G_2$ contains $H = \Aut(\mathbb{Q}_0, \leqslant)$ as a subgroup with $\mathbb{Q}_0 = \mathbb{Q} \setminus \{0 \}$. Note that $(\mathbb{Q}_0, \leqslant)$ is isomorphic to $(\mathbb{Q}, \leqslant)$. Thus by Theorem \ref{Noetherian wrt homogeneous order}, $k[\mathbb{Q}] = k[x_0][\mathbb{Q}_0]$ is Noetherian under the action of $H$, so it is also Noetherian under the highly homogenous action of $G_2$ on $\mathbb{Q}_0$. This finishes the proof.
\end{proof}

This corollary does not hold when we replace highly homogenous groups by oligomorphic groups.

\begin{example} \label{example}
The group $G = \Sym(\mathbb{N})$ is highly transitive on $\mathbb{N}$, and hence is oligomorphic on $\mathbb{N} \times \mathbb{N}$. Since $G$ can be viewed as a subgroup of $G \times G$ via the diagonal embedding, we deduce that $G \times G$ is oligomorphic on $\mathbb{N} \times \mathbb{N}$ as well. However, it is well known that $k[\mathbb{N} \times \mathbb{N}]$ is not Noetherian up to the action of $G \times G$; see \cite[Example 3.8]{HS}. Thus the condition that $G$ is oligomorphic on $S$ is not sufficient to guarantee the noetherianity of $k[S]$ up to the action of $G$.

On the other hand, there do exist groups $G$ such that the action of $G$ on $S$ is not highly homogenous but $k[S]$ is Noetherian up to this action. Indeed, for an arbitrary $x \in S$, the action of $G = \Sym(S_x)$ is not even transitive on $S$, but $k[S]$ is Noetherian up to the action of $G$.
\end{example}

\subsection{An equivalent statement of the axiom of choice}

The existence of globally homogenous linear order on $\mathbb{Q}$ (or any countable set) and $\mathbb{R}$ (or any set with the same cardinality as $\mathbb{R}$) is known. The following lemma asserts that we can construct a globally homogenous linear order for every infinite set.

\begin{lemma} \label{key lemma}
If we admit the axiom of choice, then for every infinite cardinal $\kappa$, there is a globally homogeneous linear order of size $\kappa$.
\end{lemma}

\begin{proof}
Since the property that every nontrivial open interval is isomorphic to the whole set is not a first order sentence, instead we prove a stronger conclusion: every infinite set $S$ can be equipped with a linearly ordered field structure $(S, +, \cdot, 0, 1, <, h)$ satisfying the additional global homogeneity: for every $a < b$ and $c < d$ in $S$, there exists $p, q \in S$ such that $x \mapsto px+q$ is an order-preserving bijection from $(a, b)$ to $(c, d)$; $h$ is an order-preserving bijection from $(0, 1)$ to $S$. Clearly, this strengthened global homogeneity is first order definable. Hence, every structure of above theory is also a globally homogenous linear order. Since, for some suitable $h$, $(\mathbb{Q}, +, \cdot, <, h)$ is a countable model of this theory by Cantor's theorem, the desired conclusion follows from L\"{o}wenheim–Skolem theorem (see \cite[Section 2.3]{Marker}).
\end{proof}

To prove Theorem \ref{main result 3}, we need to introduce a notion in set theory called the \textit{Hartogs number}. Explicitly, we define
\[
\aleph(X) = \{\alpha \text{ is an ordinal } \mid \text{ there is an injection from $\alpha$ to } X \},
\]
which is the least ordinal that can not be injectively mapped into $X$. Note that the existence of this number only relies on Zermelo-Fraenkel set theory {\rm ZF} alone and does not require the axiom of choice; for details, see \cite{Hartogs}.

\begin{theorem} \label{homogeneous linear order is possible}
The axiom of choice is equivalent to the statement that every infinite set admits a globally homogenous linear order.
\end{theorem}

\begin{proof}
Suppose that the axiom of choice holds, and let $S$ be an infinite set with cardinality $\kappa$. By Lemma \ref{key lemma}, one can construct a globally homogenous linear order $\leqslant$ on another set $S'$ with the same cardinality. Then one can define a linear order on $S$ via an arbitrary bijection from $S$ to $S'$, which is clearly globally homogenous.

Now suppose that every infinite set can be equipped with a globally homogeneous linear order. We want to to show that every infinite set admits a well order, which is equivalent to the axiom of choice. Fix an infinite set $X$, and set $\kappa = \aleph(P(X))$ where $P(X)$ is the power set of $X$. Without loss of generality, we may assume that $X \cap \kappa = \emptyset$.

Fix a globally homogeneous linear order $\leqslant$ on $X \cup \kappa$, and define a map $f: \kappa \rightarrow P(X)$ by
\[
f(\alpha) = \{x \in X \mid x < \alpha\}.
\]
By the definition of $\kappa$, $f$ cannot be an injection, so there are $\alpha < \beta$ such that $f(\alpha) = f(\beta)$. This happens if and only if the interval $(\alpha, \beta)$ is disjoint from $X$, so $(\alpha, \beta) \subseteq \kappa$. Since $\kappa$ is an ordinal, it is well orderable, so is $(\alpha, \beta)$. But since $\leqslant$ is globally homogeneous, one can find an isomorphism $\rho: (\alpha, \beta) \to (X \cup \kappa, \leqslant)$. Using this bijection and a well order $\preccurlyeq$ on $(\alpha, \beta)$, one can define a well order on $X \cup \kappa$, which induces a well order on $X$.
\end{proof}

\end{document}